\pdfoutput=1
\documentclass[11pt,reqno]{amsart}
\usepackage[paperheight=279mm,paperwidth=18cm,textheight=26cm,textwidth=14cm,includehead]{geometry}
\usepackage[T1]{fontenc}
\usepackage[utf8]{inputenc}
\usepackage[unicode]{hyperref}
\hypersetup{
bookmarks=true,
colorlinks=true,
citecolor=[rgb]{0,0,0.5},
linkcolor=[rgb]{0,0,0.5},
urlcolor=[rgb]{0,0,0.75},
pdfpagemode=UseNone,
pdfstartview=FitH,
pdfdisplaydoctitle=true,
pdftitle={Bi-parameter Carleson embeddings with product weights},
pdfauthor={Arcozzi, Mozolyako, Psaromiligkos, Volberg, Zorin-Kranich},
pdflang=en-US
}

\usepackage[style=alphabetic,maxalphanames=4]{biblatex}
\usepackage{amssymb}
\usepackage{amsmath}
\usepackage{amsthm}
\usepackage{amsfonts}
\usepackage{mathtools}

\usepackage{tikz}
\usetikzlibrary{patterns}

\DeclarePairedDelimiterX\Set[1]\{\}{%

#1
}

\DeclarePairedDelimiterX\innerp[2]{\langle}{\rangle}{#1,#2}

\numberwithin{equation}{section}
\theoremstyle{plain}
\newtheorem{theorem}[equation]{Theorem}

\newtheorem{lemma}[equation]{Lemma}

\theoremstyle{remark}
\newtheorem{remark}[equation]{Remark}

\theoremstyle{definition}

\newcommand{\al}{\alpha}
\newcommand{\eps}{\varepsilon}

\newcommand{\vf}{\varphi}

\newcommand{\Om}{\Omega}

\newcommand{\cI}{\mathcal I}
\newcommand{\cJ}{\mathcal J}

\newcommand{\cM}{\mathcal M}

\newcommand{\cP}{\mathcal P}

\newcommand{\bC}{\mathbb C}

\newcommand{\bE}{\mathbb E}

\newcommand{\bT}{\mathbb T}

\newcommand{\bi}{\mathbf{i}}
\newcommand{\bj}{\mathbf{j}}
\newcommand{\bk}{\mathbf{k}}

\begin{document}

\title[Sidon constant]%
{An estimate of Sidon constant for complex polynomials with unimodular coefficients}
\author[A.~Volberg]{Alexander Volberg}
\thanks{AV is partially supported by the NSF grants  DMS-1900268 and DMS-2154402}
\address[A.~Volberg]{Department of Mathematics, Michigan Sate University,  and Hausdorff Center, Universit\"at Bonn}
\email{volberg@math.msu.edu}
%
%
\begin{abstract}
In this paper we are concerned with the Bohnenblust--Hille type inequalities for certain polynomials of bounded degree but of very large number of variables. As the polynomials will be defined on groups, one can think about the problem as the estimate of Sidon constants. In most cases the sharp constants are unknown. 
We estimate the universal constant   concerning the Sidon type estimates of degree $d$ polynomials of $n$ variables $z_1, \dots, z_n$ with unimodular coefficients. For polynomials that have constant absolute value of coefficients, this allows us to  improve slightly the estimate from 
\cite{BPS}, see also \cite{DGMS}. The main result is Theorem \ref{better} below.
\end{abstract}
\maketitle

\section{Preliminaries and notations}

Given a $d$-linear form  on $\bC^n\times \dots \bC^n=\bC^{nd}$
$$
B(z^1,\dots, z^d)=\sum_{i_1, \dots i_d=1}^n b_{i_1,\dots, i_d} z^1_{i_1}\dots, z^d_{i_d}=\sum_{i}^n b_{i} z^1_{i_1}\dots, z^d_{i_d}
$$
a famous inequality (Bohnenblust--Hille inequality in modern form) for forms claims
\begin{equation}
\label{form}
\Big(\sum_i |b_i|^{\frac{2d}{d+1}}\Big)^{\frac{d+1}{2d}} \lesssim d^{a} \max_{z^s\in \bT^n, s=1, \dots, d} |B(z^1, \dots, z^d)|\,,
\end{equation}
where $a$ is an absolute constant $a\le 0.39$.
Let
$$
{\mathcal M}(d, n)=\{i=(i_1,\dots, i_d): 1\le i_k \le n\}.
$$

There is a vast literature devoted to the estimate of homogeneous multilinear polynomials on $\bC^n$,  of degree $d$:
$$
P(z) =P(z_1, \dots , z_n) =\sum a_\bi z_\bi,
$$
where $\bi=(\bi(1),\dots, \bi(d))$ is a non-decreasing function from $[d]=\{1, 2, \dots, d\}$ into $[n]= \{1,2, \dots, n\}$: $\bi:[d]\to [n]$. A map $\bi$ can be viewed as:
$$
\cJ(d, n)=\{1\le i_1\le\dots\le i_d\le n\}.
$$
The cardinality of pre-images $\al_k =|\bi^{-1}(k)|=\sharp\{s: i_s=k\}$, $k=1, \dots, n$, $s=1, \dots, m$, $0\le \al_k \le m$, $\al_1+\dots+\al_n=m$.

This estimate enjoyed a lot of attention with and without extra condition on coefficients, see, for example, the following papers and the literature cited therein: \cite{CDS}, \cite{DS}, \cite{DGMS} and \cite{DMP}. The latter work is concerned with polynomials on group $\{-1,1\}^n$ (Hamming cube) rather than on $\bT^n$.

\medskip

It is easy to calculate the cardinality of $\cJ(d, n)$. It is 
\begin{equation}
\label{J}
\sharp \cJ(d, n) = {n+d-1 \choose d}\,.
\end{equation}
One can see it in \cite{DGMS} formula (2.71).

Notice that $\cM(d, n)$ can be associated with all maps $[d]$ to $[n]$.
Every $j\in {\mathcal M}(d, n)$ lies in equivalent class of a certain $\bi\in \cJ(d, n)$ by permutation, $j\in [\bi]$, and the number of elements in the equivalence class is (it is easy to see, but also see \cite{DGMS}, page 72)
$$
|[\bi]| = \frac{d!}{|\{k: i_k=1\}|!\dots |\{k: i_k=n\}|!}.
$$
In other words, for the map $\bi=(i_1, i_2\dots i_m)$, $1\le i_1\le i_2\le\dots\le i_m\le n$, we can consider   such $j=(j_1, \dots, j_m)$ that are equivalent to $\bi$, meaning that there exists a permutation
$$
\sigma : [d]\to [d],
$$
such that $i_1=j_{\sigma(1)}, \dots, i_d = j_{\sigma(d)}$.

\medskip

We want to list some form(s) such that $B_P(z,z, \dots, z) =P(z), z=(z_1, \dots, z_n)$. There are many of them.
Of course we can think that $\{a_\bi\}$ themselves are coefficients of the form, meaning
\begin{equation}
\label{main}
a_i = a_\bi, \text{if}\,\, i_1=\bi(1), \dots, i_d= \bi(d),\,\, \bi\in \cJ(d, n); \,\, a_i =0\,\, \text{otherwise}\,.
\end{equation}
In other words, if the reader sees the symbol $a_\bi$, this means that $\bi\in \cJ(d, n)$ (or this coefficients is zero).
We get a form $B_P$ with relatively few non-zero coefficients and we can write
$$
P(z)= B_P(z, z, \dots, z)\,.
$$
Another option is to consider the symmetric form $B_P^s$ with coefficients $b_i$, then
\begin{equation}
\label{mains}
|[\bi]| b_j = a_\bi, \text{if}\,\, j \in [\bi], \bi\in \cJ(d, n)\,.
\end{equation}
Instead of  estimate \eqref{form}   one then is interested in the following considerably more involved estimate
\begin{equation}
\label{hpoly}
\Big(\sum_{\bi\in \cJ(d, n)} |a_\bi|^{\frac{2d}{d+1}}\Big)^{\frac{d+1}{2d}} \lesssim C(d) \max_{z\in \bT^n} |B_P(z, \dots, z)|= C(d) \max_{x\in \bT^n} |P(z)|\,,
\end{equation}

The right hand side of \eqref{hpoly} is considerably smaller than the right hand side of \eqref{form}. In fact we take maximum of the form only on the diagonal in \eqref{hpoly}), hence the constant $C(d)$ might not be of polynomial type $d^C$. However, this is unknown.
The best known constant, which is denoted by $BH^{=d}$, see \cite{DGMS}, Chapter 10, is
\begin{equation}
\label{dmp}
BH^{=d} =C(d) \le C_1 e^{C_0 \sqrt{d\log \,d}}\,.
\end{equation}

Even in the case when all $a_\bi$ have the same absolute value it was not known whether one can give a better estimate (to the best of our knowledge).

Next we consider the case of unimodular coefficients:
$$
|a_\bi|=1, \,\bi\in \cJ(d, n),
$$
and of course all other $a$ vanish.

\begin{theorem}
\label{better}
For  degree $d$  polynomials with unimodular coefficients we have
a slightly better estimate:
$$
BH^{=d} \le C_\tau e^{ \big(\frac{\sqrt{2}}{\sqrt{3}}+\tau\big) \sqrt{d\log d}  }\,\|P\|_{L^\infty(\bT^n)}, \quad \forall \tau>0\,.
$$
\end{theorem}
This is slightly better than the estimate that one can derive from \cite{BPS} or \cite{DGMS}, which is
$$
BH^{=d} \le C e^{ \frac{4}{\sqrt{3}} \sqrt{d\log d} +\frac{8}{\sqrt{3}} \sqrt{\frac{d}{\log d}}  }\,\|P\|_{L^\infty(\bT^n)},
$$
but of course a slightly better estimate is only for polynomials with unimodular coefficients.



\section{Proof of Theorem \ref{better}}
\label{proofcomplex}
Let
$$
A:=\big(\sum_{\bi\in \cJ(d, n)} |a_\bi|^{\frac{2d}{d+1}} \big)^{\frac{d+1}{2d}} \,.
$$
In our case of unimodular coefficients we have from \eqref{J} that
$$
A= {n+d-1\choose d}^{\frac12 +\frac1{2d}}\,.
$$

We will need the following easy (and basically sharp) estimate
\begin{equation}
\label{strange}
1\le k \le d\le n, \quad {n+d-1\choose d}^{\frac{1}{d}} \le {n+k-1\choose k}^{\frac{1}{k}}\,.
\end{equation}
Hence
\begin{equation}
\label{strange2}
1\le k \le d\le n, \quad {n+d-1\choose d}^{\frac{d+1}{2d}} \le {n+d-1\choose d}^{\frac{1}{2}}{n+k-1\choose k}^{\frac{1}{2k}}\,.
\end{equation}
It can be proved by hand, but below we give another proof.

\bigskip

We use first Blei's estimate  \cite{Blei}, \cite{DNP} for the form whose non-zero coefficients sits only on $\bi\in \cJ(d, n)$, and those coefficients are $a_\bi$, $\bi\in \cJ(d, n)$, and zero otherwise. Arithmetic mean/geometric mean inequality and H\"older's inequality give
\begin{eqnarray}
\label{ai}
&A=\big(\sum_{\bi\in \cJ(d, n)} |a_\bi|^{\frac{2d}{d+1}} \big)^{\frac{d+1}{2d}} \le^{\text{Blei}} \notag
\\
&\Big\{ \Pi_{S\subset [d], |S|=d-k}\Big[\sum_{\bi\in \cJ(S, n)} \big(\sum_{\bj \in \cJ(\hat S, n): \bi\oplus \bj \in \cJ(d, n)|}|a_{\bi\oplus \bj}|^2\Big)^{\frac12 \cdot \frac{2k}{k+1}}\Big]^{\frac{k+1}{2k}}\Big\}^{{ d \choose k}^{-1}}\!\!\! \!\!\le 
\frac{|\cJ(k, n)|}{{d \choose k}}\times \notag
\\
&\sum_{S\subset [d], |S|=d-k}\Big( \frac1{|\cJ(\hat S, n)|}\sum_{\bi\in \cJ(S, n)} (\sharp\,\text{non-decreasing extensions of}\,\, \bi) ^{\frac{k}{k+1}}\Big)^{\frac{k+1}{2k}} \le   \notag
\\
& {n+k-1 \choose k}^{\frac{k+1}{2k}}{d\choose k }^{-1}\times             \notag
\\
&\sum_{S\subset [d], |S|=d-k}  \Big(\frac1{\sharp \cJ(\hat S, n)}\sum_{\bj \in \cJ(\hat S, n)} (\sharp\,\text{\,non-decreasing extensions of\,} \bj)^{\frac{k}{k+1}}\Big)^{\frac{k+1}{2k}} \le \notag
\\
& {n+k-1 \choose k}^{\frac{k+1}{2k}}{d\choose k }^{-1}\times  \notag
\\
&\sum_{S\subset [d], |S|=d-k}  \Big(\frac1{\sharp \cJ(\hat S, n)}\sum_{\bj \in \cJ(\hat S, n)} (\sharp\,\text{\,non-decreasing extensions of\,} \bj)\Big)^{\frac12}
\end{eqnarray}

Now we need to see what is the expectation of number of non-decreasing extension of non-decreasing random $\bj: \hat S\to [n]$.
\begin{lemma}
\label{ext}
Let $S\subset ]d], |S|=d-k$.
Then
$$
\bE_{\bj: \hat S}(\sharp{\,\text{\,non-decreasing extensions of\,} \bj})={n+k-1 \choose k} ^{-1}{n+d-1 \choose d}\,.
$$
\end{lemma}
\begin{proof}
Let us denote temporarily  randomly chosen non-decreasing function on $\hat S$ by $\vf$ and non-decreasing functions on $[d]$ by $\Phi$. Fix $S$ that is  $S\subset [d], |S|=d-k$. Then we want to see what is
$$
\bE_{\bj: \hat S}(\sharp{\,\text{\,non-decreasing extensions of\,} \bj})=
$$
$$
{n+k-1 \choose k}^{-1} \sum_{\vf}\sharp\text{\,non-decreasing extensions of\,} \vf =
$$
$$
{n+k-1 \choose k} ^{-1} \sum_{\vf} \sum_{\Phi: \Phi|\hat S=\vf} 1  = {n+k-1 \choose k} ^{-1}\sum_{\Phi}\sum_{\vf: \vf=\Phi|\hat S} 1 =  
$$
$$
{n+k-1 \choose k} ^{-1}\sum_{\Phi} 1 = {n+k-1 \choose k} ^{-1}{n+d-1 \choose d}\,.
$$
The fourth equality is because there is only one such $\vf$ for every given $\Phi$.
\end{proof}

Hence
\begin{eqnarray}
\label{hence}
&{{n+k-1 \choose k}}^{\frac{k+1}{2k}} \bE_{S} \Big(\bE_{\bj: \hat S}(\sharp{\,\text{\,non-decreasing extensions of\,} \bi|S})\Big)^{\frac12} =       \notag
\\
&{{n+k-1 \choose k}}^{\frac{k+1}{2k}}  \Big(  {n+k-1 \choose k} ^{-1}{n+d-1 \choose d}\Big)^{\frac12} =  {n+k-1 \choose k}^{\frac1{2k}}
{n+d-1 \choose d}^{\frac12}\,.
\end{eqnarray}

Together \eqref{ai} and \eqref{hence} give
\begin{eqnarray}
\label{A1}
A:= \big(\sum_{\bi\in \cJ(d, n)} |a_\bi|^{\frac{2d}{d+1}} \big)^{\frac{d+1}{2d}} \le  {n+k-1 \choose k}^{\frac1{2k}}
{n+d-1 \choose d}^{\frac12}\,.
\end{eqnarray}
This is \eqref{strange2}, by the way.

\subsection{The estimate of expression $B$}
\label{B}

Below we write expression that imitates the expression on page 264 of \cite{DGMS} (that follows \cite{BPS}). We will estimate it from above and from below. Then we will compare it with expression $A$ from \eqref{A1}. 

Remind that $b_\bi$ are coefficients of symmetric form $B_P^s$ such that $B_P^s(z,\dots, z)= P(z)$.

\begin{equation}
\label{B}
B:=\Big\{\Pi_{S\subset [d], |S|=k}\Big[ \sum_{\bi \in\cM(S, n)}\Big(\sum_{\bj \in \cM(\hat S, n)} |[\bi \oplus \bj]| |b_{\bi \oplus\bj}|^2\Big)^{\frac12\cdot \frac{2k}{k+1}}\Big]^{\frac{k+1}{2k}} \Big\}^{1/{d \choose k}}
\end{equation}
To estimate from {\it above} we imitate \cite{DGMS}. One first notice a simple inequality
\begin{equation}
\label{ij}
\frac{|[\bi \oplus \bj]|}{|[\bj]|} \le d(d-1)\dots (d-k+1),\quad |\hat S|=d-k,
\end{equation}
then, using the calculation on page 265 of \cite{DGMS}, one gets
\begin{eqnarray}
\label{Babove}
 &B\le \big[ d(d-1)\dots (d-k+1)\big]^{\frac12} \times    \notag
 \\
 &\Big\{\Pi_{S\subset [d], |S|=k}\Big[ \sum_{\bi \in\cM(S, n)}\Big(\sum_{\bj \in \cM(\hat S, n)} |[ \bj]| |b_{\bi \oplus\bj}|^2\Big)^{\frac12\cdot \frac{2k}{k+1}}\Big]^{\frac{k+1}{2k}} \Big\}^{1/{d \choose k}}\le       \notag
 \\
 & \big[ d(d-1)\dots (d-k+1)\big]^{\frac12}\cdot  HC^{d-k}_{\frac{2k}{k+1}, 2} \cdot BH^k_{\text{form}}\cdot \frac{(d-k)! d^d}{(d-k)^{d-k} d!} \cdot \|P\|_\infty =       \notag
 \\
 &\big[ d(d-1)\dots (d-k+1)\big]^{-\frac12}\cdot  HC^{d-k}_{\frac{2k}{k+1}, 2} \cdot BH^k_{\text{form}}\cdot \frac{d^d}{(d-k)^{d-k} } \cdot \|P\|_\infty
\end{eqnarray}

Here $BH^k_{\text{form}} \le k^{c_0}$, $c_0<0.4$, is the  known estimate of the Bohnenblust--Hille type for $k$-forms, see Theorem 6.13 of \cite{DGMS}. The constant $HC^{d-k}_{\frac{2k}{k+1}, 2} $ is the $L^2/ L^{\frac{2k}{k+1}}$ hypercontractivity constant for homogeneous polynomials of degree $d-k$. It can be estimated as follows (see \cite{DGMS}, Remark 8.16):
\begin{equation}
\label{HC}
HC^{d-k}_{\frac{2k}{k+1}} \le e^{\frac{4(d-k)}{3k-1}}\,.
\end{equation}

\bigskip

Now we give the estimate of $B$ from {\it below}. We will use that $|[\bk]||b_\bk|= |a_{r(\bk)}|=1$, where $r(\bk)$ means the non-decreasing rearrangement of any given (not necessarily non-decreasing) $\bk:[d]\to [n]$.
\begin{eqnarray}
\label{Bbelow}
&B= \Big\{\Pi_{S\subset [d], |S|=k}\Big[ \sum_{\bi \in\cM(S, n)}\Big(\sum_{\bj \in \cJ(\hat S, n)} |[\bi \oplus \bj]|\cdot|[\bj]| |b_{\bi \oplus\bj}|^2\Big)^{\frac12\cdot \frac{2k}{k+1}}\Big]^{\frac{k+1}{2k}} \Big\}^{1/{d \choose k}} \ge^{\eqref{ij}}\notag
\\
&\frac{1}{[d(d-1)\dots (d-k+1)]^{\frac12}} \Big\{\Pi_{S\subset [d], |S|=k}\Big[ \sum_{\bi \in\cM(S, n)}\Big(\sum_{\bj \in \cJ(\hat S, n)} |[\bi \oplus \bj]|^2 |b_{\bi \oplus\bj}|^2\Big)^{\frac12\cdot \frac{2k}{k+1}}\Big]^{\frac{k+1}{2k}} \Big\}^{1/{d \choose k}} =\notag
\\
& \frac{1}{[d(d-1)\dots (d-k+1)]^{\frac12}} \Big\{\Pi_{S\subset [d], |S|=k}\Big[ \sum_{\bi \in\cM(S, n)}\Big(\sum_{\bj \in \cJ(\hat S, n)} 1\Big)^{\frac12\cdot \frac{2k}{k+1}}\Big]^{\frac{k+1}{2k}} \Big\}^{1/{d \choose k}} = \notag
\\
&\frac{1}{[d(d-1)\dots (d-k+1)]^{\frac12}} \Big\{\Pi_{S\subset [d], |S|=k}  \,  |\cJ(\hat S, n)|^{\frac12}  \cdot |\cM(S, n)|^{\frac{k+1}{2k}}       \Big\}^{1/{d \choose k}} =  \notag
\\
&\frac{1}{[d(d-1)\dots (d-k+1)]^{\frac12}}   |\cJ(d-k, n)|^{\frac12}  \cdot |\cM(k, n)|^{\frac{k+1}{2k}}  = {n+d-k-1\choose  d-k}^{\frac12}\cdot \big(n^k\big)^{\frac{k+1}{2k}} \ge  \notag
\\
&\frac{1}{[d(d-1)\dots (d-k+1)]^{\frac12}} {n+d-k-1\choose  d-k}^{\frac12}\cdot (k! )^{\frac{k+1}{2k}} \cdot {n \choose k}^{\frac{k+1}{2k}}
\end{eqnarray}

We need also to exchange here ${n \choose k}$ by ${n+k-1 \choose k}$. We promise that everywhere in what follows
\begin{equation}
\label{promise}
k\le C\sqrt{d}\,.
\end{equation}
Under this condition (and of course $d\le \sqrt{n}$, see \eqref{largen})
\begin{eqnarray}
\label{compa}
&\frac{{n+k-1\choose k}}{{n\choose k}} = \frac{(n+k-1)! \cdot (n-k)!}{(n-1)!\cdot n!}=         \notag
\\
& \frac{(n+k-1)\dots n}{n\dots (n-k+1)} \le \frac{(n+k)^k}{(n-k)^k} \le  e^{\frac{3k}{n} k} \le C\,.
\end{eqnarray}
The last inequality is by our promise \eqref{promise}. 

Using \eqref{compa} we can extend \eqref{Bbelow} to ($c= C^{-1}$):
\begin{eqnarray}
\label{Bbelow1}
&B= \Big\{\Pi_{S\subset [d], |S|=k}\Big[ \sum_{\bi \in\cM(S, n)}\Big(\sum_{\bj \in \cJ(\hat S, n)} |[\bi \oplus \bj]||[\bj]| |b_{\bi \oplus\bj}|^2\Big)^{\frac12\cdot \frac{2k}{k+1}}\Big]^{\frac{k+1}{2k}} \Big\}^{1/{d \choose k}} \ge\notag
\\
& c \frac{1}{[d(d-1)\dots (d-k+1)]^{\frac12}} {n+d-k-1\choose  d-k}^{\frac12}\cdot (k! )^{\frac{k+1}{2k}} \cdot {n+k-1 \choose k}^{\frac{k+1}{2k}}
\end{eqnarray}

\subsection{Estimate of $\frac{A}{B}$}
\label{AB}

\begin{eqnarray}
\label{AbyB}
&\frac{A}{B [d(d-1)\dots (d-k+1)]^{\frac12}} \le  C\frac{{n+k-1 \choose k}^{\frac1{2k}}{n+d-1 \choose d}^{\frac12}}{  {n+d-k-1\choose  d-k}^{\frac12}\cdot (k! )^{\frac{k+1}{2k}} \cdot {n+k-1 \choose k}^{\frac{k+1}{2k}}} =            \notag
\\
&C\frac{ 
{n+d-1 \choose d}^{\frac12}}{  {n+d-k-1\choose  d-k}^{\frac12}\cdot (k! )^{\frac{k+1}{2k}} \cdot {n+k-1 \choose k}^{\frac{1}{2}}} \le                      \frac{ C}
 {(k! )^{\frac{1}{2}} } \Big(\frac{{n+d-1 \choose d}}{{n+d-k-1\choose  d-k}\cdot  {n+k-1 \choose k}}\Big)^{\frac12}\le \notag
 \\
 & \frac{1}{{d\choose k}^{1/2}} \frac{ C}
 {(k! )^{\frac{1}{2}} } \Big( \frac{(n-1)!}{(n+k-1)!}\cdot \frac{(n+d-1)!}{(n+d-k-1)!}\Big)^{\frac12} \le   \frac{1}{{d\choose k}^{1/2}}\frac{ C}
 {(k! )^{\frac{1}{2}} } \Big( \frac{(n+d-1)\cdot\dots \cdot(n+d -k )}{(n+k-1)\cdot\dots\cdot n}\Big)^{\frac12}  \le     \notag
 \\
 & \frac{1}{{d\choose k}^{1/2}}  \frac{ C}
 {(k! )^{\frac{1}{2}} }\Big( (1+\frac{d-k}{n+k-1})\cdot \dots\cdot (1+ \frac{d-k}{n})\Big)^{\frac12} \le  \frac{1}{{d\choose k}^{1/2}}  \frac{ C}
 {(k! )^{\frac{1}{2}} } e^{\frac12 \frac{(d-k) k}{n}}.     \notag
 \end{eqnarray}
 So,
 \begin{eqnarray}
 &\frac{A}{B}\le [d(d-1)\dots (d-k+1)]^{\frac12} \frac{1}{{d\choose k}^{1/2}} \frac{ C'}
 {(k! )^{\frac{1}{2}} },
\end{eqnarray}
if we consider the case of large $n$:
\begin{equation}
\label{largen}
d^{2} \le c\, n\,.
\end{equation}
This is the only interesting regime, if \eqref{largen} does not occur, then the estimate we wish to get is very simple.
This last statement can be checked easily. Suppose $n\le d^2$. Let $\cP_u(d, n)$ be the family of homogeneous polynomials of degree $d$ in $n$ variables with all coefficients being $1$ by absolute value. Let $P\in \cP_u(d, n)$.
Then ${n+d-1\choose d}^{\frac12}=\|P\|_{L^2(\bT^n)} \le \|P\|_{L^\infty(\bT^n)}$, and we conclude that
$$
{n+d-1\choose d}^{\frac12} \le \inf_{P\in  \cP_u(d, n)}\|P\|_{L^\infty(\bT^n)}\,.
$$
But then
$$
\big(\sum |a_\bi|^{\frac{2d}{d+1}}\big)^{\frac{d+1}{2d}} = {n+d-1\choose d}^{\frac12+\frac{1}{2d}}\le  {n+d-1\choose d}^{\frac{1}{2d}}\cdot \inf_{P\in  \cP_u(d, n)}\|P\|_{L^\infty(\bT^n)}\,.
$$
As $n\le d^2$ we are left to notice an easy inequality
$
{ d^2 +d-1\choose d} \le e^d (d+1)^d,
$
which finally gives
$$
\big(\sum |a_\bi|^{\frac{2d}{d+1}}\big)^{\frac{d+1}{2d}} = {n+d-1\choose d}^{\frac{1}{2d}}\inf_{P\in  \cP_u(d, n)}\|P\|_{L^\infty(\bT^n)}\le \sqrt{ e(d+1)}\cdot \inf_{P\in  \cP_u(d, n)}\|P\|_{L^\infty(\bT^n)}\,.
$$
This is a polynomial in $d$ estimate, and we proved  Theorem \ref{better} in the non-interesting regime $n\le d^2$. Only \eqref{largen} regime is interesting (as it was expected). For this regime we finish the proof now.

\section{The end of the proof of the main theorem}
\label{pr}

Now we can estimate $A$. Gathering \eqref{AbyB} and \eqref{Babove}, we get
\begin{eqnarray}
\label{A2}
&A =\frac{A}{B}\cdot B \le  \frac{ C}
 {(k! )^{\frac{1}{2}} }\frac{ [d(d-1)\dots (d-k+1)]^{\frac12}}{\big[d\cdot \dots \cdot (d-k+1)\big]^{\frac12}} \cdot B\le          \notag
 \\
 & \frac{ C}
 {(k! )^{\frac{1}{2}}\big[d\cdot \dots \cdot (d-k+1)\big]^{\frac12}}  HC^{d-k}_{\frac{2k}{k+1}, 2} \cdot BH^k_{\text{form}}\cdot \frac{d^d}{(d-k)^{d-k} } \cdot \|P\|_\infty \le              \notag
 \\
 & \frac{ C\,d^{\frac{k}2}}
 {(k! )^{\frac{1}{2}}\big[d\cdot \dots \cdot (d-k+1)\big]^\frac12}  HC^{d-k}_{\frac{2k}{k+1}, 2} \cdot BH^k_{\text{form}}\cdot \big(1+\frac{k}{d-k}\big)^{d-k}\cdot d^{\frac{k}2} \cdot \|P\|_\infty \le               \notag
 \\
 &C\Big( (1-\frac1{d})\dots (1-\frac{k-1}{d})\Big)^{-\frac12}k!^{-\frac12}\cdot e^{\frac{4(d-k)}{3k}} \cdot k^{c_0}\cdot e^k\cdot d^{\frac{k}2} \cdot \|P\|_\infty\le                                                    \notag
 \\
 & e^{\big( \frac1{d}+\dots +\frac{k-1}{d}\big)}\cdot e^{\frac{4d}{3k}+2k-k\log k} \cdot d^{\frac{k}{2}}\cdot \|P\|_\infty\le        \notag
 \\
 &k^{c_0}  e^{\frac{ck^2}{d}}\cdot e^{\frac{4d}{3k}+\frac12 k\log d+ 2k -k\log k} \cdot \|P\|_\infty\,.
\end{eqnarray}
We used  $1-x\ge e^{-2x}$ for sufficiently small positive $x$ and \eqref{promise}.

Integer $k$ was free to choose with the promise \eqref{promise}. So let us choose $k=\frac{2\sqrt{2}}{\sqrt{3}}\sqrt{\frac{d}{\log d}}$.
Then we get a final estimate that follows and we prove Theorem \ref{better}:
\begin{equation}
\label{final}
\big(\sum_{\bi\in \cJ(d, n)} |a_\bi|^{\frac{2d}{d+1}} \big)^{\frac{d+1}{2d}} \le C_\tau e^{ \Big(\frac{\sqrt{2}}{\sqrt{3}}+\tau\big) \sqrt{d\log d}   }\,\|P\|_{L^\infty(\bT^n)}\,, \quad \forall \tau>0,
\end{equation} 
which is slightly better than the estimate that would follow from \cite{BPS}:
$$
\big(\sum_{\bi\in \cJ(d, n)} |a_\bi|^{\frac{2d}{d+1}} \big)^{\frac{d+1}{2d}} \le C e^{ \frac{4}{\sqrt{3}} \sqrt{d\log d} +\frac{8}{\sqrt{3}} \sqrt{\frac{d}{\log d}}  }\,\|P\|_{L^\infty(\bT^n)}\,.
$$
But of course we have only polynomials with unimodular coefficients in \eqref{final}.

So far polynomials were homogeneous. Consider now non-homogeneous case.
Let $P(z)=\sum_{|\al|\le d} c_\al z^\al$ is any polynomial of $n$ variables of degree $d$ (not necessarily homogeneous) such that $|c_\al|=1$ for all $\al$. Then
$$
\Big(\sum_{|\al|\le d} |c_\al |^{\frac{2d}{d+1}}\Big)^{\frac{d+1}{2d}} \le C  e^{ \frac{2\sqrt{2}}{\sqrt{3}} \sqrt{d\log d} +\frac{4\sqrt{2}}{\sqrt{3}} \sqrt{\frac{d}{\log d}}  }\|P\|_{L^\infty(\bT^n)}\,.
$$
One uses \cite{DGMS}, page 132: consider $Q(z, w):=\sum_{|\al|\le d} c_\al z^\al w^{d-|\al|}$, where $w\in \bC$. Applying Theorem \ref{better} to this polynomial of $n+1$ variables we prove the corollary because
$$
 \|P\|_{L^\infty(\bT^n)}=  \|Q\|_{L^\infty(\bT^{n+1})}\,.
 $$
 

\section{Modifying \cite{DMP}. Strengthening Blei's inequality for unimodular coefficients}
\label{dmp}

Now we wish to make the same improvement of the constant $BH^{=d}_{\text{poly}}$ not for  degree $d$  homogeneous polynomials of $z=(z_1,\dots, z_n)$ but for  degree $d$  homogeneous polynomials of $x=(x_1, \dots, x_n)$, where $x_k=\pm 1$ (boolean case).

The approach is the same as above.

We follow \cite{DMP} and in one place we use
\begin{equation}
\label{abs}
|a_\bi|\equiv |a_\bj|, \,\forall \bi, \bj\,.
\end{equation}

\medskip

First one uses Blei's inequality, which follows by careful use of complex interpolation. It holds for any $d$-linear form, so, in particular for forms with coefficients $\{a_\bi\}$ which have many zero coefficients.
\begin{equation}
\label{blei}
\Big(\sum_{\bi \in J(d, n)}|a_\bi|^{\frac{2d}{d+1}}\Big)^{\frac{d+1}{2d}} \le \Big[ \prod_{S\subset [d], |S|=d-k} \Big(\sum_{\bj\in J(\hat S, n)}\{\sum_{\bi\in J(S, n)} |a^\sharp_{\bi\oplus \bj}|^2\}^{\frac12\cdot\frac{2k}{k+1}}\Big)^{\frac{k+1}{2k}}\Big]^{{d\choose k }^{-1}}
\end{equation}
In this formula $\hat S = [d]\setminus S$. It is true for any $1\le k\le d$, see \cite{Blei}.

To explain what we just wrote, one needs a couple of notations:
symbol $J(S, n)$ denotes all  strictly increasing maps $\bi: S\to [n]$, where $S\subset [d]$. Symbol $J(d, n)$ is used if $S=[d]$.

Map $\bi\oplus \bj$ is defined naturally. But might very well be not a strictly increasing map from $S_1\cup S_2$ into $[n]$. So what coefficient $a_{\bi\oplus\bj}$ should be prescribed to it?  For $\bi\in J(S, n), \bj\in J(\hat S, n), \hat S=[d]\setminus S$ we define the strictly increasing rearrangement $r(\bi\oplus\bj):[d] \to [n]$ of $\bi \oplus \bj$, if it exists. We put:


\begin{equation}
\label{ext1}
a^\flat_{\bi\oplus\bj}:= \begin{cases} a_{r(\bi\oplus\bj)},\quad\text{if}\,\, \bi\oplus \bj\,\,\text{is injective},
\\
0\quad \,\, \text{otherwise}.
\end{cases}
\end{equation}
So this coefficients is zero if a strictly increasing rearrangement does not exist.

 Thus we get {\it one of the possible extensions} of function $a$ from its support $J(d, n)$ to $J(S, n)\times J(\hat S, n)$, which works simultaneously for all $S$.

For unimodular coefficients Blei's \eqref{blei} inequality can be strengthened as follows.

\begin{equation}
\label{blei2}
\Big(\sum_{\bi \in J(d, n)}|a_\bi|^{\frac{2d}{d+1}}\Big)^{\frac{d+1}{2d}} \lesssim {d\choose k}^{-\frac12} \Big[ \prod_{S\subset [d], |S|=d-k} \Big(\sum_{\bj\in J(\hat S, n)}\{\sum_{\bi \in J( S, n)} |a^\flat_{\bi\oplus\bj}|^2\}^{\frac12\cdot\frac{2k}{k+1}}\Big)^{\frac{k+1}{2k}}\Big]^{{d\choose k }^{-1}}
\end{equation}

Following \cite{DMP} we write using the $e^{(d-k)(\frac{k+1}{2k}-\frac12)}$ hypercontractivity constant for polynomials of degree $\le  d-k$ of $n$  boolean variables:
$$
\Big(\sum_{\bj\in J(\hat S, n)}\{\sum_{\bi \in J( S, n)} | a^\flat_{\bi\oplus\bj}|^2\}^{\frac12\cdot\frac{2k}{k+1}}\Big)^{\frac{k+1}{2k}}=\Big(\sum_{\bj\in J(\hat S, n)}\{ \bE_{x \in\Om_n}|\sum_{\bi \in J( S, n)} a^\flat_{\bi\oplus\bj}x_{\bi(S)}|^2\}^{\frac12\cdot\frac{2k}{k+1}}\Big)^{\frac{k+1}{2k}} \le
$$
$$
e^{(d-k)(\frac{k+1}{2k}-\frac12)}\Big(\sum_{\bj\in J(\hat S, n)} \bE_{y \in\Om_n}|\sum_{\bi \in J( S, n)} a^\flat_{\bi\oplus\bj}y_{\bi(S)}|^{\frac{2k}{k+1}}\Big)^{\frac{k+1}{2k}}  = 
$$
$$
e^{\frac{d-k}{2k}} \Big(\bE_{y \in\Om_n}\sum_{\bj\in J(\hat S, n)}|\sum_{\bi \in J( S, n)} a^\flat_{\bi\oplus\bj}y_{\bi(S)}|^{\frac{2k}{k+1}}\Big)^{\frac{k+1}{2k}}  \le 
$$
$$
e^{\frac{d-k}{2k}} \Big(\max_{y\in \Om_{n}}\sum_{\bj\in J(\hat S, n)}|\sum_{\bi \in J( S, n)} a^\flat_{\bi\oplus\bj}y_{\bi(S)}|^{\frac{2k}{k+1}}\Big)^{\frac{k+1}{2k}}  \le
$$
\begin{equation}
\label{blei3}
e^{\frac{d-k}{2k}}  BH^{=k}  \Big(\max_{y\in \Om_{n}}\max_{x\in \Om_{n}}|\sum_{\bj\in J(\hat S, n)}\sum_{\bi \in J( S, n)} a^\flat_{\bi\oplus\bj}y_{\bi(S)}x_{\bj(\hat S)}|\Big)
\end{equation}

\medskip

Below $[i]$ means strictly increasing rearrangement of injective map $i$.
  We can rewrite the last sum as follows
by introducing one more extension of function $a$ originally defined only on $J(d, n)$:
\begin{equation}
\label{ext2}
\hat a_{ij} :=  a^\flat_{[i]\oplus [j]} = \begin{cases} a_{[[i]\oplus[j]]},\,\,\text{if}\,\, [i]\oplus [j]\,\, \text{is injective}.
\\
0,\quad\text{otherwise}
\end{cases}
\end{equation}

$$
\Big(\max_{y\in \Om_{n}}\max_{x\in \Om_{n}}|\sum_{\bj\in J(\hat S, n)}\sum_{\bi \in J( S, n)} a^\flat_{\bi\oplus\bj}y_{\bi(S)}x_{\bj(\hat S)}|\Big)=
$$
$$
\frac1{k!(d-k)!} \Big(\max_{y\in \Om_{n}}\max_{x\in \Om_{n}}|\sum_{j\in \cI(\hat S, n)}\sum_{i \in \cI( S, n)}\hat a_{ij}y_ix_j|\Big)=
$$
\begin{equation}
\label{blei4}
=\frac{d!}{k!(d-k)!} \Big(\max_{y\in \Om_{n}}\max_{x\in \Om_{n}}|\sum_{j\in \cI(\hat S, n)}\sum_{i \in \cI( S, n)}c_{ij}y_ix_j|\Big),
\end{equation}
where
\begin{equation}
\label{c}
c_{ij}:= \frac{\hat a_{ij}}{d!}\,.
\end{equation}

Let, following \cite{DMP}, put
$$
L(x^1,\dots, x^d) =  \sum_{i\in \cI(d, n)} c_i x^1_{i_1}\dots x^d_{i_d}\,.
$$
Then
$$
P(x)= L(x, \dots, x)
$$
and for any $S$, $|S|=d-k$,
\begin{equation}
\label{Ldk}
\sum_{j\in \cI(\hat S, n)}\sum_{i \in \cI( S, n)}c_{ij}y_i x_j=L(x,\dots x, y, \dots y) =: L_{k}(x, y),
\end{equation}
where $x$ repeats $k$ times and $y$ repeats $d-k$ times.

In \cite{DMP} one considers
$$
p(t) := P(kt x + (d-k)y) =   L(kt x + (d-k)y, \dots, kt x + (d-k)y) = (k)^d t^d L(x, \dots, x) +\dots+
$$
$$
k^k (d-k)^{d-k} {d\choose k} t^{k} L_{k}(x, y)+\dots\,.
$$
One is required now to estimate  $L_{k}(x, y)$ (see \eqref{blei4}, \eqref{Ldk}) by the $\max_{t\in [-1,1]}|p(t)|$. If this is done, linearity of the form $L$ implies
$$
\max_{t\in [-1,1]}|p(t)| \le d^d \max_{t\in [-1,1]}\big|P\big(\frac{(d-k)t x + k y}{d}\big) \big| = d^d \max_{x\in \Om_n} |P(x)|,
$$
the last inequality follows from the convexity of $|P|$ for multilinear $P$.

Using a famous estimate of $k$-th coefficient of  degree $d$ polynomial bounded by $1$ on $[-1,1]$ we get
\begin{eqnarray}
\label{Ldk1}
&|L_{k}| \le {d\choose k}^{-1}k^{-k} (d-k)^{-(d-k)} L(d, k)\max_{t\in [-1,1]}|p(t)|\le             \notag
\\
&L(d, k)\cdot\frac{(d-k)! k! d^d}{d! k^{k} (d-k)^{d-k}}\cdot \max_{x\in \Om_n} |P(x)|\,.
\end{eqnarray}
where $L(d, k)$ is the absolute value  of $k$-th coefficient of  degree $d$ of Chebyshov polynomial $T_d$.

For $k<<d$ this Markov constant $L(d, k)$ is of order $\Big(\frac{d}{k}\Big)^{k}$.

\bigskip

Gathering all together we get using the strengthening of Blei's inequality \eqref{blei2}
$$
BH^{=d} / BH^{=k}\lesssim {d\choose k}^{-\frac12} e^{\frac{d-k}{2k}}  \frac{d^d}{k^k (d-k)^{d-k}} \Big(\frac{d}{k}\Big)^{k}=
$$
$$
{d\choose k}^{-\frac12} e^{\frac{d-k}{2k}}  \Big(1+\frac{k}{d-k}\Big)^{d-k} \Big(\frac{d}{k}\Big)^{2k}\,.
$$

For $k<< d$ (which is actually where the optimal $k$ is located) this expression is essentially
$$
\asymp  \min_{k\le d}\Big(\frac{d}{k}\Big)^{\frac32 k}e^{\frac{d}{2k}} \asymp\min_{k\le d} e^{\frac{d}{2k} +\frac32 k\log d  -\frac32 k\log k} \asymp  e^{C_0 \sqrt{d \log d}}
$$

Of course it is the same estimate as in \cite{DMP}, only constant $C_0$ became smaller as in \cite{DMP}  one gets
$$
\asymp  \min_{k\le d}\Big(\frac{d}{k}\Big)^{2 k}e^{\frac{d}{2k}} \asymp\min_{k\le d} e^{\frac{d}{2k} +2 k\log d  -2 k\log k} \asymp e^{C_1 \sqrt{d \log d}}
$$
with $C_1> C_0$.

\begin{remark}
The constant $e^{\frac{d}{2k}}$ comes from hypercontractivity and cannot be improved much. Even the assumption that $|a_\bi|=const$ cannot diminish it essentially. In fact consider
$$
f=\frac{k!}{n^{k/2}}\sum_{|S|=k, S\subset [n]} x^S 
$$
then this polynomial has the  limit distribution $n \to \infty$ that converges to the distribution of Hermite polynomial $H_k$, and this shows that the hypercontractivity constant cannot be improved.
I am grateful to P. Ivanisvili for this remark.
The constant $\Big(\frac{d}{k}\Big)^{Ak}$ comes from Markov constant. We managed to diminish $A$ for the case $|a_\bi|=const$  by noticing that one can work with a smaller intermediary sum. But it is hard to imagine that by this method one can replace this term by something much smaller, like e.g. by $\Big(\frac{d}{k}\Big)^{A\sqrt k}$.
\end{remark}


\section{Strengthening of Blei's inequality for unimodular coefficients}
\label{str}
We want to explain \eqref{blei2}.

We will compare the left hand side and the right hand side of Blei's inequality for the case when coefficients are unimodular.
 \begin{eqnarray}
\label{blei1}
&\Big(\sum_{\bi \in J(d, n)}|a_\bi|^{\frac{2d}{d+1}}\Big)^{\frac{d+1}{2d}} \le  \notag
\\
&\Big[ \prod_{S\subset [d], |S|=d-k} \Big(\sum_{\bj\in J(\hat S, n)}\{\sum_{\bi \in J( S, n)} | a^\flat_{\bi\oplus\bj}|^2\}^{\frac12\cdot\frac{2k}{k+1}}\Big)^{\frac{k+1}{2k}}\Big]^{{d\choose k }^{-1}}= \notag
\\
&\Big[ \prod_{S\subset [d], |S|=d-k} \Big(\sum_{\bj\in J(\hat S, n)}\{\sum_{\bi \in J( S, n)} |a_{r(\bi\oplus\bj)}|^2\}^{\frac12\cdot\frac{2k}{k+1}}\Big)^{\frac{k+1}{2k}}\Big]^{{d\choose k }^{-1}}
\end{eqnarray}

\bigskip

 When the unimodular property  is present, we can notice that the right hand side of \eqref{blei} is considerably smaller than the right hand side of \eqref{blei1}. Let us estimate their ratio.  All $|a_{\cdot}|$ are $1$ so we need to compare the cardinality in sums. We will use the analog of \eqref{strange2}. Of course the left hand side is just ${n\choose d}^{\frac12+\frac1{2d}}$. It is easy to verify
 $$
 1\le k\le d\le n\Rightarrow {n\choose d}^{\frac1{d}} \le {n\choose k}^{\frac1{k}}\,.
 $$
 Hence,
 \begin{equation}
 \label{elem}
 \Big(\sum_{\bi \in J(d, n)}|a_\bi|^{\frac{2d}{d+1}}\Big)^{\frac{d+1}{2d}} \le {n\choose d}^{\frac12}{n\choose k}^{\frac1{2k}}\,.
 \end{equation}

On the other hand,  if all  $|a_\bi |=1$,  then not all but most of $|a_{r(\bi\oplus \bj)}|=1$ (some are $0$ but we will take this into consideration now) the right hand side of \eqref{blei1},  is the expression
$$
 \Big[ \prod_{S\subset [d], |S|=d-k} \Big(\sum_{\bj\in J(\hat S, n)}\{\sum_{\bi \in J( S, n)} |a_{r(\bi\oplus\bj)}|^2\}^{\frac12\cdot\frac{2k}{k+1}}\Big)^{\frac{k+1}{2k}}\Big]^{{d\choose k }^{-1}}=
 $$
 $$
  \Big[ \prod_{S\subset [d], |S|=d-k} \Big(\sum_{\bj\in J(\hat S, n)}\{\sharp J(S, n; \bj)\}^{\frac12\cdot\frac{2k}{k+1}}\Big)^{\frac{k+1}{2k}}\Big]^{{d\choose k }^{-1}}=
 $$
 \begin{equation}
 \label{denom}
 {n-k  \choose d- k}^{\frac12}  \Big[ \prod_{S\subset [d], |S|=d-k} \Big(\sum_{\bj\in J(\hat S, n)}1\Big)^{\frac{k+1}{2k}}\Big]^{{d\choose k }^{-1}}=  {n -k\choose d-k}^{\frac12} {n \choose k}^{\frac{k+1}{2k}}
 \end{equation}
 
 The expression $J(S, n; \bj)$ above means the family of maps $\bi$ from $S$ to $[n]$, such that $\bi\oplus \bj$ is injective.
 Given that $|\hat S|= k$, clearly the cardinality of this set is $\sharp J(S, n; \bj)= {n- k\choose d-k}$.

Now we need to divide the expression in \eqref{elem} by the expression in \eqref{denom}, and to see how much we won.
\begin{eqnarray}
\label{win}
&\frac{ {n  \choose k}^{\frac1{2k}}
{n \choose d}^{\frac12}}{ {n-k \choose d-k}^{\frac12} {n\choose k}^{\frac{k+1}{2k}}} =\frac{{n\choose d}^{\frac12}} { {n-k \choose d- k}^{\frac12} {n\choose k}^{\frac12}} =
\\
&\Big(\frac{n! \, k!\, (n-d)!\,(d-k)!\, (n-k)!}{d! \,(n-d)!\,(n-k)!\,n!}\Big)^{\frac12}=\frac{1}{{d \choose k}^{\frac12}}\,.
\end{eqnarray}

This is what we win in comparison with \cite{DMP} if coefficients are unimodular.

 
\section{What about Aaronson--Ambainis conjecture for homogeneous polynomials with constant modulus of coefficients on hypercube $\{-1, 1\}^n$?}
\label{AA}
The conjecture is that there are universal constants $C, K$ such that  independently of the number of variables $n$ for every (homogeneous) polynomial of degree $d$ the following holds:
\begin{equation}
\label{AAconj}
\Big(\frac{Var[f]}{d}\Big)^K \le C\max_{1\le i \le n} \text{Inf}_i[f] \cdot \|f\|_\infty^{2K-2}\,.
\end{equation}
Now let us reformulate this for $f$ having $|\hat f(S)|=1$ for all $S$, $|S|=d$.

Denote
$$
T_{n, d}: =\inf_{f\in \cP_d(1)}\|f||_\infty\,.
$$

{\bf Trivial estimate} of $T_{n, d}(1)$  follows from $f\in \cP_d(1)\Rightarrow \|f\|_2 ={ n+d-1 \choose d}^{1/2}$,
$\|f\|_2 \le  T_{n, d}$, so,
\begin{equation}
\label{triv}
{ n\choose d}^{1/2} \le T_{n, d}\,.
\end{equation}
But \eqref{AAconj} for $f\in \cP_d(1)$ immediately reduces to
$$
{n\choose d}^K \le C \frac{d^{K+1}}{n} { n \choose d} T_{n, d}^{2K-2}\,.
$$
In other words, \eqref{AAconj} becomes the following estimate on $T_{n, d}$:
\begin{equation}
\label{AA1}
\exists K:\,{n \choose d}^{1/2} \le C \frac{d^{\frac{K+1}{2K-2}}}{n^{\frac{1}{2K-2}} } T_{n, d}\,.
\end{equation}
Obviously the trivial estimate \eqref{triv} entails the latter estimate
for this regime of $(d, n)$:
\begin{equation}
\label{trivReg}
d^{K+1}\ge n, \, \text{or}\,\,d\ge n^\eps, \,\, K=\frac{1}{\eps}-1\,.
\end{equation}

What about the opposite regime? Let us  write a non-trivial Sidon type estimate.
Suppose we have for $f\in \cP_d(1)$
\begin{equation}
\label{S1}
\|\hat f\|_{\frac{2d}{d+1}}  \le S(d) \|f\|_\infty\,.
\end{equation}

Then for $f\in  \cP_d(1)$ we obtain another estimate on $T_{n, d}$ from below--this time a non-trivial one:
\begin{equation}
\label{non-triv1}
{n \choose d}^{\frac12 +\frac1{2d}} \le S(d) T_{n, d}\,.
\end{equation}
 
  We first notice that for $d\le n^{1/2}$ (and we only interested in $d\le n^{1/2}$  as for the opposite case we are in the trivial regime \eqref{trivReg} with $\eps =1/2$)
one has 
$$
{n \choose d} \ge C n^d d^{-d-1/2}\Rightarrow {n \choose d}^{1/2d} \ge c\, \frac{n^{1/2}}{d^{1/2}}\,.
$$
If we {\bf assume} that $S(d) \le d^A$ with some universal $A$, then using the last display and \eqref{non-triv1} we conclude that we would have
$$
\frac{n^{1/2}}{d^{1/2}} {n\choose d}^{1/2} \le C\,d^A T_{n, d}\,.
$$
But this implies \eqref{AA1} for the regime $d\le  n^{\eps_0}$, where $\eps_0$ can be seen from inequality

$$
d^{A+\frac12 -\frac{K+1}{2K-2}} \le n^{\frac12-\frac1{2K-2}},
$$
where $K$ is, say $=3$.

\medskip

However, at this moment we do not have $S(d)\le d^A$. What is known is 
$$
S(d) \le e^{C_0 \sqrt{d\log d}}\,.
$$
Then we have
\begin{equation}
\label{non-triv2}
d\le n^{1/2}\Rightarrow n^{1/2} e^{-C_0\sqrt{d\log d}} {n \choose d}^{1/2} \le T_{n, d}\,.
\end{equation}
This estimate beats the trivial \eqref{triv}  in the regime
$$
d\lesssim \frac{\log^2 n}{\log\log n},
$$
which far short of $d\le n^{\eps_0}$. In other words, to get to \eqref{AA1}, which we need, we start with \eqref{non-triv2}, plug $T_{n,d}$ estimate from it and conclude that  ${n\choose  d}^{1/2} \le C \frac{d^{\frac{K+1}{2K-2}}}{n^{\frac{1}{2K-2}} } T_{n, d}$ holds we need to conclude the following:
$$ 
{n\choose k}^{1/2}  \lesssim  \frac{d^{\frac{K+1}{2K-2}}}{n^{\frac{1}{2K-2}} } n^{1/2} e^{-C_0 \sqrt{d \log d}} {n\choose k}^{1/2} 
$$
for some $K$. But this conclusion can be done only in the regime
$$
d\lesssim \frac{\log^2 n}{\log\log n}.
$$
We have a big gap in regimes, where nothing is known:
$$
 \frac{\log^2 n}{\log\log n}<< d \le n^\eps\,.
 $$
 
 \section{Discussion}
 \label{disc}
 
 Why the complex case turned out to be slightly better (in this approach) than the Hamming case?
 The reason is as follows. Consider   complex  polynomials $P(z_1, z_2)$ and real polynomials 
 $p(x, y)$. In both cases these are homogenous polynomials of degree $d$. And we assume $\|P\|_{L^\infty(\bT^2)} \le 1$ and $\|p\|_{L^\infty([-1,1]^2)} \le 1$. Now the estimate of the absolute values coefficient $A(k, d)$ in front of $z_1^k z_2^{d-k}$ is obviously: $|A(k, d)|\le 1$. 
 
 At the same time, the estimates of of the absolute values coefficient $a(k, d)$ in front of $x^k y^{d-k}$ is  related to the $k$-th coefficient of Chebyshov polynomial $T_d$, which was called $L(d, k)$ above, and which is
 of order $\Big(\frac{d}{k}\Big)^k$ when $k<<d$.

\end{document}